\documentclass[a4paper,11pt]{article}
\usepackage{amsmath,amsthm,amssymb,enumitem,xcolor}
\usepackage{tikz}

\usepackage[hang]{footmisc}
\usepackage[nosort,nocompress,noadjust]{cite}

\usepackage[bookmarks=false,hyperfootnotes=false,colorlinks,
    linkcolor={red!60!black},
    citecolor={blue!50!black},
    urlcolor={blue!80!black}]{hyperref}

\renewcommand{\eqref}[1]{\hyperref[#1]{(\ref{#1})}}

\newlist{enumlist}{enumerate}{1}
\setlist[enumlist]{labelindent=0cm,label=\arabic*.,ref=\arabic*,labelwidth=2.5ex,labelsep=0.5ex,leftmargin=3ex,align=left,topsep=0.5ex,itemsep=1ex,parsep=1ex}

\newlist{itemlist}{itemize}{1}
\setlist[itemlist]{labelindent=0cm,label=$\bullet$,labelwidth=2.5ex,labelsep=0.5ex,leftmargin=3ex,align=left,topsep=0.5ex,itemsep=1ex,parsep=1ex}

\numberwithin{equation}{section}

{\theoremstyle{definition}\newtheorem{definition}{Definition}[section]
\newtheorem*{definition*}{Definition}

\newtheorem{remark}[definition]{Remark}

\newtheorem*{example*}{Example}
\newtheorem*{examples*}{Examples}}

\newtheorem{proposition}[definition]{Proposition}
\newtheorem{lemma}[definition]{Lemma}
\newtheorem{theorem}[definition]{Theorem}

{\theoremstyle{definition}}

\newcommand{\Gh}{\widehat{G}}
\newcommand{\recht}{\rightarrow}
\newcommand{\cU}{\mathcal{U}}
\newcommand{\muh}{\widehat{\mu}}
\newcommand{\Z}{\mathbb{Z}}
\newcommand{\C}{\mathbb{C}}
\newcommand{\R}{\mathbb{R}}
\newcommand{\Prob}{\operatorname{Prob}}

\newcommand{\lambdah}{\widehat{\lambda}}
\newcommand{\om}{\omega}

\newcommand{\vN}{\operatorname{vN}}
\newcommand{\F}{\mathbb{F}}
\newcommand{\FAW}{\mathord{\text{\rm FAW}}}
\newcommand{\cK}{\mathcal{K}}
\newcommand{\cH}{\mathcal{H}}
\newcommand{\Probc}{\operatorname{Prob}_{\text{\rm c}}}
\newcommand{\ot}{\otimes}
\newcommand{\Ba}{B_{\text{\rm a}}}
\newcommand{\cD}{\mathcal{D}}
\newcommand{\dpr}{^{\prime\prime}}
\newcommand{\cF}{\mathcal{F}}

\newcommand{\vphi}{\varphi}

\newcommand{\Aut}{\operatorname{Aut}}
\newcommand{\Out}{\operatorname{Out}}
\newcommand{\Inn}{\operatorname{Inn}}
\newcommand{\N}{\mathbb{N}}
\newcommand{\al}{\alpha}
\newcommand{\be}{\beta}
\newcommand{\Hom}{\operatorname{Hom}}
\newcommand{\Ad}{\operatorname{Ad}}
\newcommand{\mutil}{\widetilde{\mu}}
\newcommand{\muop}{\mu^{\text{\rm op}}}
\newcommand{\id}{\mathord{\text{\rm id}}}
\newcommand{\Autpmp}{\operatorname{Aut}_{\text{\rm pmp}}}
\newcommand{\cZ}{\mathcal{Z}}
\newcommand{\ovt}{\mathbin{\overline{\otimes}}}
\newcommand{\Sd}{\operatorname{Sd}}
\newcommand{\actson}{\curvearrowright}

\begin{document}

\begin{center}
{\boldmath\Large\bf Non-classification of free Araki-Woods factors and $\tau$-invariants}

\bigskip

{\sc by Rom\'{a}n Sasyk\footnote{\noindent Universidad de Buenos Aires, Departamento de Matem\'{a}tica and Instituto Argentino de Matem\'{a}ticas-CONICET, Buenos Aires (Argentina). E-mail: rsasyk@dm.uba.ar. Supported in part by research grant PICT 2012-1292 (ANPCyT).}, Asger T\"{o}rnquist\footnote{\noindent University of Copenhagen, Department of Mathematics, Copenhagen (Denmark). E-mail: asgert@math.ku.dk. Supported in part by the DNRF Niels Bohr Professorship of Lars Hesselholt, and by DFFResearch Project 7014-00145B.} and Stefaan Vaes\footnote{\noindent KU~Leuven, Department of Mathematics, Leuven (Belgium). E-mail: stefaan.vaes@kuleuven.be. Supported by European Research Council Consolidator Grant 614195 RIGIDITY, and by long term structural funding~-- Methusalem grant of the Flemish Government.}}
\end{center}

\begin{abstract}\noindent
We define the standard Borel space of free Araki-Woods factors and prove that their isomorphism relation is not classifiable by countable structures. We also prove that equality of $\tau$-topologies, arising as invariants of type III factors, as well as cocycle and outer conjugacy of actions of abelian groups on free product factors are not classifiable by countable structures.
\end{abstract}

\section{Introduction}

In his celebrated work, Connes \cite{Co75} proved that all \emph{amenable} factors acting on a separable Hilbert space are hyperfinite and deduced that there is a unique amenable factor for each of the types I$_n$ with $n \in \N$, II$_1$, II$_\infty$ and III$_\lambda$ with $\lambda \in (0,1)$. The uniqueness of the hyperfinite III$_1$ factor was proved in Haagerup's fundamental paper \cite{Ha85}. In \cite{Co75,Co74b}, Connes also proved that hyperfinite type III$_0$ factors are isomorphic to Krieger factors. So, using \cite{Kr74,CT76}, they are completely classified by their \emph{flow of weights},
which is a properly ergodic action of $\R$ on a standard measure space. This last result can equally well be interpreted as a non-classification theorem. Indeed, properly ergodic flows up to conjugacy cannot be classified in any concrete way, and descriptive set theory provides a framework to give a precise meaning to the complexity of such a classification problem. In particular, using Hjorth's concept of \emph{turbulence} \cite{Hj97}, it follows from \cite{FW03} that properly ergodic flows and therefore amenable type III$_0$ factors cannot be classified by countable structures like countable groups or other combinatorial invariants. Even for the subclass of infinite tensor products of matrix algebras, such a non-classification result holds, see \cite{ST09}.

For \emph{nonamenable} factors, complete classification theorems for large families of such factors were obtained in Popa's deformation/rigidity theory \cite{Po06}. In particular, his work allows to ``embed'' unclassifiable structures (like, for instance, probability measure preserving transformations) into the theory of II$_1$ factors and prove that II$_1$ factors are not classifiable by countable structures, see \cite{ST08}.

Among the most natural and most studied nonamenable factors are the \emph{free Araki-Woods factors} of Shlyakhtenko \cite{Sh96}, associated with an orthogonal representation $(U_t)_{t\in\R}$ of $\R$ on a real Hilbert space $\cH_\R$ of dimension at least $2$. This construction generalizes Voiculescu's free Gaussian functor. When $U_t$ is the trivial representation, the associated free Araki-Woods factor is of type II$_1$ and isomorphic to the free group factor $L(\F_n)$ with $n = \dim_\R(\cH_\R)$. When $U_t$ is non trivial, one obtains a full factor of type III that may be viewed as a free probability analog of the classical Araki-Woods factors.

The \emph{almost periodic} free Araki-Woods factors, associated with almost periodic representations $U_t$, were completely classified by Shlyakhtenko in \cite{Sh96} in terms of Connes' $\Sd$-invariant, which is a countable subgroup of $\R_*^+$. While the complete classification of arbitrary free Araki-Woods factors remains wide open, a first result in that direction, classifying a large family of \emph{non almost periodic} free Araki-Woods factors, was obtained in \cite{HSV16}. As the first main result of our paper, we turn the space of free Araki-Woods factors into a standard Borel space and reinterpret one of the examples in \cite{HSV16} as providing a Borel family $M_\mu$ of free Araki-Woods factors indexed by arbitrary non-atomic probability measures on the Cantor set $C$ such that $M_\mu \cong M_\eta$ if and only if the measures $\mu$ and $\eta$ are mutually absolutely continuous; see Theorem \ref{thm.non-classif-free-araki-woods}. In combination with \cite{KS99}, it follows that free Araki-Woods factors are not classifiable by countable structures.

The finest modular theory invariant for full type III factors is \emph{Connes $\tau$-invariant} \cite{Co74a}, defined as the weakest topology on $\R$ that makes the canonical modular homomorphism $\R \recht \Aut(M)/\Inn(M)$ continuous. In \cite{Sh02}, Shlyakhtenko proved that there is a continuum of possible $\tau$-topologies and used this to construct a continuum of non almost periodic free Araki-Woods factors. As our second main result, we prove in Theorem \ref{thm.non-classifiable-tau} that $\tau$-topologies are not even classifiable by countable structures. This provides another proof for the non-classification of free Araki-Woods factors, but also applies to other classification questions where a $\tau$-invariant makes sense. As an example, we prove in Theorem \ref{thm.non-class-actions} that \emph{cocycle or outer conjugacy} of actions of $\Z$, $\R$, or any other non compact, abelian, locally compact group $G$ on a free group factor $L(\F_n)$, or any other free product factor $M_1 * M_2$, are not classifiable by countable structures. Finally note that a similar non-classification theorem for \emph{conjugacy} of a single automorphism (i.e.\ an action of $\Z$) was obtained in \cite{KLP08} for the hyperfinite II$_1$ factor $R$ and in \cite{KLP14} for free product factors, using a spectral invariant. Such spectral invariants are however not preserved under cocycle conjugacy. Moreover, by \cite{Oc85}, all outer actions of a discrete amenable group $G$ on $R$ are cocycle conjugate, and this leads to a complete classification of actions of $G$ on $R$ up to cocycle conjugacy.

{\bf Acknowledgment.} Research for this paper was mainly carried out when R.\ Sasyk was at the Hausdorff Research Institute for Mathematics in Bonn during the Trimester program on ``von Neumann Algebras''. He wishes to thank the Hausdorff Institute for kind hospitality and support. S.\ Vaes wants to thank D.\ Shlyakhtenko for discussions and suggestions that helped to improve this article.



\section{Non-classification of free Araki-Woods factors}\label{non-classification}

Whenever $\cK$ is a separable Hilbert space, we denote by $\vN(\cK)$ the set of all von Neumann subalgebras of $B(\cK)$ and we equip $\vN(\cK)$ with the Effros Borel structure defined in \cite{Ef64} as the smallest $\sigma$-algebra such that for every $\vphi \in B(\cK)_*$, the map $\vN(\cK) \recht \R : M \mapsto \|\vphi|_M\|$ is measurable. We start by defining a Borel subset $\FAW \subset \vN(\cK)$ that contains, up to isomorphism, all free Araki-Woods factors. The correct interpretation of \cite[Example 5.4]{HSV16} then provides a Borel map $\Theta : \Probc(C) \recht \FAW$ from the space of continuous probability measures on the Cantor set $C$ to the space of free Araki-Woods factors such that $\Theta(\mu) \cong \Theta(\mu')$ if and only $\mu \approx \mu'$, where $\approx$ denotes the equivalence relation of mutual absolute continuity on $\Prob(K)$. So, $\Theta$ is a Borel reduction from $(\Probc(K),\approx)$ to $(\FAW,\cong)$ and it follows from \cite[Theorem 2.1]{KS99} that free Araki-Woods factors are not classifiable by countable structures.

To define the Borel set $\FAW$, fix a separable, infinite dimensional Hilbert space $\cH$. Define the full Fock space $\cK = \cF(\cH)$ given by
\begin{equation}\label{eq.full-fock}
\cF(\cH) = \C \Omega \oplus \bigoplus_{n = 1}^\infty \underbrace{\bigl(\cH \ot \cdots \ot \cH\bigr)}_{\text{$n$ times}} \; .
\end{equation}
We denote by $\Ba(\cH)$ the space of antilinear bounded operators from $\cH$ to $\cH$ and we equip both $\Ba(\cH)$ and $B(\cH)$ with the standard Borel structure induced by the strong topology. We define the following \emph{parameter space} for $\FAW \subset \vN(\cK)$.
\begin{equation}\label{eq.def-D}
\begin{split}
\cD = \bigl\{ (J,T) \in \Ba(\cH) \times B(\cH) \bigm| & J = J^* , J^2 = 1, T = T^*, 0 \leq T \leq 1, T + JTJ = 1,\\ &\text{$T$ and $1-T$ have trivial kernel}\; \bigr\} \; .
\end{split}
\end{equation}
Since the operators with trivial kernel form a Borel subset of $B(\cH)$, it is easy to see that $\cD$ is a Borel subset of $\Ba(\cH) \times B(\cH)$. The formula $S = J(T^{-1}-1)^{1/2}$ defines a bijection between the set $\cD$ and the set of all densely defined, closed, antilinear operators $S$ on $\cH$ satisfying $S^2 \subset 1$, with the inverse being determined by the polar decomposition of $S$. To every such involution $S$ is associated the free Araki-Woods factor acting on $\cK$ given by
\begin{equation}\label{eq.map-Phi}
\Phi(J,T) = \bigl\{\ell(\xi) + \ell(S(\xi))^*  \bigm| \xi \in D(S)\bigr\}\dpr = \bigl\{ \ell(T^{1/2}\xi) + \ell(J(1-T)^{1/2} \xi)^* \bigm| \xi \in \cH \bigr\}\; .
\end{equation}
It follows that $\Phi : \cD \recht \vN(\cK)$ is a Borel map. Since the graph of $S$ can be recovered from $M=\Phi(J,T)$ as the closure of the set
$$\bigl\{ (x \Omega, x^* \Omega) \bigm| x \in M, x \Omega \in \cH, x^* \Omega \in \cH \bigr\} \; ,$$
it follows that $\Phi$ is injective.

\begin{definition}
For $\cK = \cF(\cH)$, we define the Borel set $\FAW \subset \vN(\cK)$ as the image of the set $\cD$ defined in \eqref{eq.def-D} under the injective Borel map $\Phi$ defined in \eqref{eq.map-Phi}.
\end{definition}

By construction, every free Araki-Woods factor is isomorphic with an element of $\FAW$. We can now deduce from \cite[Example 5.4]{HSV16} the following result.

\begin{theorem}\label{thm.non-classif-free-araki-woods}
Let $C$ be the Cantor set. There exists a Borel map $\Theta : \Probc(C) \recht \FAW$ such that $\Theta(\mu) \cong \Theta(\mu')$ if and only if $\mu \approx \mu'$, where $\approx$ denotes the equivalence relation of mutual absolute continuity. The map $\Theta$ can be chosen such that all $\Theta(\mu)$ have as $\tau$-invariant the usual topology on $\R$.

In particular, the equivalence relation $(\FAW,\cong)$ is not classifiable by countable structures.
\end{theorem}

\begin{proof}
Denote by $\lambda$ the Lebesgue measure on the interval $[0,1]$. Fix the separable Hilbert space
$$\cH = L^2([0,1],\lambda) \oplus L^2([0,1],\lambda) \oplus \C^2 \; .$$
Denote by $e_0,e_1 \in \C^2$ the two standard basis vectors. Define the anti-unitary involutions $J_0 : \C^2 \recht \C^2$ and $J : \cH \recht \cH$ given by
$$J_0(e_1) = e_2 \;\; , \;\; J_0(e_2) = e_1 \quad\text{and}\quad J(\xi_1 \oplus \xi_2 \oplus \xi_3) = \overline{\xi_2} \oplus \overline{\xi_1} \oplus J_0(\xi_3) \; .$$
Fix any $q \in (0,1)$. Using the functions $\vphi_\mu : [0,1] \recht [0,1]$ introduced in Lemma \ref{lem.map-vphi} below, we define for every probability measure $\mu \in \Prob([0,1])$, the operator $T_\mu \in L^\infty([0,1]) \oplus L^\infty([0,1]) \oplus M_2(\C) \subset B(\cH)$, given by
$$T_\mu = \bigl(1+\exp(\vphi_\mu)\bigr)^{-1} \oplus \bigl(1+\exp(-\vphi_\mu)\bigr)^{-1} \oplus \begin{pmatrix} (1+q)^{-1} & 0 \\ 0 & (1+q^{-1})^{-1}\end{pmatrix} \; .$$
By construction, $(J,T_\mu) \in \cD$ and by Lemma \ref{lem.map-vphi}, the map $\mu \mapsto (J,T_\mu)$ is Borel. We define the Borel map
$$\Theta_1 : \Prob([0,1]) \recht \FAW : \Theta_1(\mu) = \Phi(J,T_\mu) \; .$$
Denoting by $\delta_a$ the Dirac measure in $a \in \R$, we associate to any nonatomic probability measure $\mu$ on $[0,1]$ the following probability measures on $\R$~: the opposite measure $\muop$ given by $\muop(\cU) = \mu(-\cU)$ and the symmetric probability measure $\mutil$ given by
$$\mutil = \frac{1}{4}\bigl( \mu + \muop + \delta_{\log q} + \delta_{-\log q}\bigr) \; .$$
Using the unitary $V : L^2([0,1],\mu) \recht L^2([0,1],\lambda)$ given by Lemma \ref{lem.map-vphi}, it follows that the free Araki-Woods factor $\Phi(J,T_\mu)$ is isomorphic with the free Araki-Woods factor denoted by $\Gamma(\mutil,1)$ in \cite{HSV16}.

By \cite[Theorems 5.1.4 and 5.2.2]{Ru62} (see also \cite[19.1 and 19.2]{Ke95}), we can choose a copy of the Cantor set $C \subset [0,1]$ such that $C$ is independent as a subset of $\R$. Viewing $\Probc(C)$ as a subset of $\Probc([0,1])$, we define
$$\Theta : \Probc(C) \recht \FAW : \Theta(\mu) = \Theta_1((\mu + \lambda)/2) \; .$$
By \cite[Example 5.4]{HSV16}, we get that for all $\mu,\mu' \in \Probc(C)$, the von Neumann algebras $\Theta(\mu)$ and $\Theta(\mu')$ are isomorphic if and only if the measures $\mu$ and $\mu'$ are mutually absolutely continuous, and that the $\tau$-invariant of every $\Theta(\mu)$ equals the usual topology on $\R$.

By \cite[Theorem 2.1]{KS99}, the equivalence relation $\approx$ of mutual absolute continuity on $\Probc(C)$ is not classifiable by countable structures. A fortiori, the equivalence relation $\cong$ on $\FAW$ is not classifiable by countable structures.
\end{proof}

The following lemma is probably well known. It is for instance essentially contained in \cite[Proof of Theorem 17.41]{Ke95}. For completeness, we provide a short proof.

\begin{lemma}\label{lem.map-vphi}
For every probability measure $\mu \in \Prob([0,1])$, define the increasing function
$$\vphi_\mu : [0,1] \recht [0,1] : \vphi_\mu(y) = \min \{x \in [0,1] \mid \mu([0,x]) \geq y\} \; .$$
Equip $\Prob([0,1])$ with the weak topology and denote by $\lambda$ the Lebesgue measure on $[0,1]$. Then, the following holds.
\begin{enumlist}
\item Each $\vphi_\mu$ is a Borel function.
\item For every $\mu \in \Prob([0,1])$, we have that $(\vphi_\mu)_*(\lambda) = \mu$.
\item The map $\Prob([0,1]) \recht L^1([0,1],\lambda) : \mu \mapsto \vphi_\mu$ is continuous.
\item If $\mu$ has no atoms, the map $\vphi_\mu$ is strictly increasing and
$$V : L^2([0,1],\mu) \recht L^2([0,1],\lambda) : (V \xi)(y) = \xi(\vphi_\mu(y))$$
is unitary.
\end{enumlist}
\end{lemma}
\begin{proof}
1.\ Since $\vphi_\mu$ is increasing, it is a Borel function.

2.\ For every $a \in [0,1]$, we have that $\vphi_\mu^{-1}([0,a]) = [0,\mu([0,a])]$. Therefore, $(\vphi_\mu)_*(\lambda)$ and $\mu$ coincide on $[0,a]$ for all $a \in [0,1]$. This implies that $(\vphi_\mu)_*(\lambda) = \mu$.

3.\ A direct computation gives that for all $\mu,\eta \in \Prob([0,1])$, we have
$$\|\vphi_\mu - \vphi_\eta\|_1 = \int_0^1 \bigl|\mu([0,a]) - \eta([0,a])\bigr| \, da \; .$$
It follows that $\Prob([0,1]) \recht L^1([0,1]) : \mu \mapsto \vphi_\mu$ is continuous.

4.\ When $\mu$ has no atoms, the function $x \mapsto \mu([0,x])$ is continuous and thus attains all values in $[0,1]$. It follows that $\vphi_\mu$ is strictly increasing. Defining $R_\mu = \vphi_\mu([0,1])$, we have that $\mu([0,1] \setminus R_\mu) = 0$ and $\vphi_\mu : [0,1] \recht R_\mu$ is a bijection sending $\lambda$ to $\mu$. It follows that $V$ is a well defined unitary operator.
\end{proof}

%

\section{\boldmath Non-classification of $\tau$-topologies}\label{sec.non-classification-tau-topologies}

Recall that a factor $M$ is called \emph{full} when $\Inn(M) \subset \Aut(M)$ is closed. Then, $\Out(M)$ has a natural Polish group structure. In \cite[Definition 5.1]{Co74a}, Connes defined the invariant $\tau(M)$ as the weakest topology on $\R$ that makes the (canonical) modular homomorphism $\R \recht \Out(M) : t \mapsto \sigma_t^\vphi$ continuous. Given a unitary representation $\pi : \R \recht \cU(\cH)$ on a separable Hilbert space, in \cite[Theorem 5.2]{Co74a}, a full factor $M$ is constructed such that $\tau(M)$ equals the weakest topology on $\R$ that makes $\pi$ continuous.

Let $G$ be a second countable, locally compact, abelian group. Since unitary representations of $G$ on a separable Hilbert space are classified by their spectral measure and multiplicity function, it is natural to define as follows the $\tau$-topology $\tau(\mu)$ for any Borel probability measure $\mu$ on the Pontryagin dual $\Gh$. Consider the unitary representation
\begin{equation}\label{eq.rep-pi-mu}
\pi_\mu : G \recht \cU(L^2(\Gh,\mu)) : (\pi_\mu(g) \xi)(\om) = \om(g) \, \xi(\om) \quad\text{for all}\;\; g \in G, \xi \in L^2(\Gh,\mu), \om \in \Gh
\end{equation}
and denote by $\tau(\mu)$ the weakest topology on $G$ that makes the map $g \mapsto \pi_\mu(g)$ continuous from $G$ to the unitary group equipped with the strong topology. We denote by $\muh$ the Fourier transform of $\mu$ defined as
$$\muh : G \recht \C : \muh(g) = \int_{\Gh} \om(g) \, d\mu(\om) \; .$$
Note that $g_n \recht e$ in the topology $\tau(\mu)$ if and only if $\muh(g_n) \recht 1$.

Given a Polish space $K$, we always equip $\Prob(K)$ with the weak topology, i.e.\ the weakest topology making the maps $\mu \mapsto \int_K F \, d\mu$ continuous for all bounded continuous functions $F \in C_b(K)$.

Finally, denote by $S_\infty$ the Polish group of all permutations of $\N$. Recall that an equivalence relation $E$ on a Polish space $X$ is said to be generically $S_\infty$-ergodic if for any continuous action of $S_\infty$ on a Polish space $Y$, any Baire measurable 
morphism $f : X \recht Y$ from $E$ to the orbit equivalence relation of $S_\infty \actson Y$
must map a comeager subset of $X$ to a single $S_\infty$-orbit. If moreover $E$ has meager orbits, it follows that $E$ is not classifiable by countable structures.


In \cite[Theorem 2.3]{Sh02}, it was proven that there is a continuum of different $\tau$-topologies on $\R$. We now prove that equality of $\tau$-topologies is not even classifiable by countable structures.

\begin{theorem}\label{thm.non-classifiable-tau}
Let $G$ be a second countable, locally compact, noncompact, abelian group. Let $K \subset \Gh$ be a nonempty closed subset for which there exists a probability measure $\mu_1 \in \Prob(\Gh)$ with support $K$ such that $\widehat{\mu_1}$ tends to zero at infinity.

Then, the equivalence relation on $\Prob(K)$ defined by $\mu \sim \mu'$ iff $\tau(\mu) = \tau(\mu')$ has meager equivalence classes and is generically $S_\infty$-ergodic. In particular, $\sim$ is not classifiable by countable structures.
\end{theorem}
Note that the theorem applies whenever $\lambdah(K)>0$, where $\lambdah$ is the Haar measure on $\Gh$, because it then suffices to choose $\mu_1 \in \Prob(K)$ in the same measure class as $\lambdah|_K$. So, for $G$ as in Theorem \ref{thm.non-classifiable-tau}, the equivalence relation on $\Prob(\Gh)$ given by equality of $\tau$-topologies is not classifiable by countable structures. In particular, it is not smooth, meaning that it is impossible to define a standard Borel space of topologies on $G$ in such a way that $\mu \mapsto \tau(\mu)$ becomes a Borel map.

\begin{proof}
Denote by $\approx$ the equivalence relation on $\Prob(K)$ given by mutual absolute continuity. Note that $\mathord{\approx} \subset \mathord{\sim}$. Since $K$ is the support of the probability measure $\mu_1$ whose Fourier transform $\widehat{\mu_1}$ tends to zero at infinity, it follows in particular that $\mu_1$ has no atoms, so that $K$ has no isolated points. By \cite[Theorem 2]{KS99}, the equivalence relation $\approx$ is generically $S_\infty$-ergodic. A fortiori, the equivalence relation $\sim$ is generically $S_\infty$-ergodic.

To prove that $\sim$ has meager equivalence classes, choose an atomic $\mu_0 \in \Prob(K)$ such that the set of atoms of $\mu_0$ is dense in $K$.

{\bf Claim.} If $g_n \in G$ is such that $g_n \recht \infty$ and $\widehat{\mu_0}(g_n) \recht 1$ and if $a \in [0,1]$, then
\begin{equation}\label{eq.myset}
\bigl\{ \mu \in \Prob(K) \bigm| \;\text{there exists a subsequence $g_{n_k}$ such that $\widehat{\mu}(g_{n_k}) \recht a$}\;\bigr\}
\end{equation}
is comeager in $\Prob(K)$.

Note that the set in \eqref{eq.myset} equals
$$\bigcap_{k \geq 1} \bigcap_{n_0 \geq 1} P(k,n_0) \quad\text{where}\quad P(k,n_0) = \bigcup_{n \geq n_0} \bigl\{ \mu \in \Prob(K) \bigm| |\widehat{\mu}(g_n) - a| < 1/k \bigr\} \; .$$
Since $P(k,n_0)$ is a union of open sets, to prove the claim, it suffices to prove that $P(k,n_0)$ is dense in $\Prob(K)$. For $i \in \{0,1\}$, denote by $P_i \subset \Prob(K)$ the set of probability measures that are absolutely continuous w.r.t.\ $\mu_i$. Since both $\mu_0$ and $\mu_1$ have support $K$, it follows that $P_i \subset \Prob(K)$ is dense for every $i \in \{0,1\}$. Choose $\eta_i \in P_i$ and define $\eta = a \eta_0 + (1-a) \eta_1$. Since $\widehat{\mu_0}(g_n) \recht 1$, we get that $\pi_{\mu_0}(g_n) \recht 1$ strongly and thus, $\widehat{\eta_0}(g_n) \recht 1$. Since $\widehat{\mu_1}$ tends to zero at infinity, the representation $\pi_{\mu_1}$ is mixing and it follows that $\widehat{\eta_1}(g_n) \recht 0$. So, $\widehat{\eta}(g_n) \recht a$ and $\eta \in P(k,n_0)$. Thus, $aP_0 + (1-a)P_1 \subset P(k,n_0)$. Since $P_i \subset \Prob(K)$ is dense for every $i \in \{0,1\}$, it follows that $P(k,n_0) \subset \Prob(K)$ is dense as well.

Assume that the equivalence class $[\mu]_\sim$ of some $\mu \in \Prob(K)$ is not meager. Since the closure of $\pi_{\mu_0}(G)$ is compact while $G$ is noncompact, we can choose as follows a sequence $g_n \in G$ such that $g_n \recht \infty$ and $\pi_{\mu_0}(g_n) \recht 1$ strongly. Starting with any sequence $h_n \in G$ such that $h_n \recht \infty$, after passage to a subsequence, we may assume that $\pi_{\mu_0}(h_n)$ converges strongly. Taking a sequence $k_n \in \N$ that tends to infinity sufficiently fast, the sequence $g_n = h_{k_n} h_n^{-1}$ has the required properties.

In particular, $\widehat{\mu_0}(g_n) \recht 1$. Since the nonmeager set $[\mu]_\sim$ intersects every comeager set, by the claim above, there exists $\eta_0 \in [\mu]_\sim$ and a subsequence $g_{n_k}$ such that $\lim_k \widehat{\eta_0}(g_{n_k}) = 1$. Applying the claim above to the sequence $(g_{n_k})$, there exists $\eta_1 \in [\mu]_\sim$ and a subsequence $g_{n_{k_l}}$ such that $\lim_l \widehat{\eta_1}(g_{n_{k_l}}) = 0$. It follows that $g_{n_{k_l}} \recht e$ in the topology $\tau(\eta_0)$, while $g_{n_{k_l}} \not\recht e$ in the topology $\tau(\eta_1)$. Since $\tau(\eta_0) = \tau(\mu) = \tau(\eta_1)$, we have reached a contradiction. So we have proven that the equivalence relation $\sim$ has meager equivalence classes.
%
\end{proof}

\begin{remark}
Let $M$ be the free Araki-Woods factor associated with an orthogonal representation $(U_t)_{t \in \R}$ of $\R$ on a real Hilbert space $\cH_\R$ of dimension at least $2$. In \cite[Corollary 8.6]{Sh97}, it is proven that $M$ is full and that $\tau(M)$ coincides with the weakest topology on $\R$ that makes the map $\R \recht O(\cH_\R) : t \mapsto U_t$ continuous (see \cite[Th\'{e}or\`{e}me 2.7]{Va04} for the fully general case). So also from Theorem \ref{thm.non-classifiable-tau}, one can deduce that free Araki-Woods factors cannot be classified by countable structures.
\end{remark}

In the case where $G = \R$, the non-classification of $\tau$-topologies in Theorem \ref{thm.non-classifiable-tau} can be made concrete as follows. Define the compact space $K = \{-1,0,1\}^\N$ and the continuous map
$$\theta : K \recht \R : \theta(x) = \sum_{n=0}^\infty x_n 2^{-n} \; .$$
Consider the Polish space $X = (0,1/4)^\N$ and define for every $a \in X$, the probability measure $\nu_a$ on $K$ given by
$$\nu_a = \prod_{n=1}^\infty \left( (1-a_n) \delta_0 + \frac{a_n}{2} \delta_{-1} + \frac{a_n}{2} \delta_1 \right) \; .$$
Note that $a \mapsto \nu_a$ is a continuous map from $X$ to $\Prob(K)$ equipped with the weak topology. Pushing forward with $\theta$, we define the probability measure $\mu_a$ on $\R$ given by $\mu_a = \theta_*(\nu_a)$. The following proposition describes the $\tau$-topology of the measure $\mu_a$.

\begin{proposition}\label{prop.concrete-tau}
For every $a \in X$, the topology $\tau(\mu_a)$ is induced by the translation invariant metric $d_a$ on $\R$ defined by
\begin{equation}\label{eq.def-metric-da}
d_a(t,s) = \left(\sum_{n=0}^\infty a_n \, d_\Z\bigl(2^{-n}(t-s)\bigr)^2 \right)^{1/2} \; ,
\end{equation}
where $d_\Z$ denotes the distance of a real number to $\Z \subset \R$.
\end{proposition}
\begin{proof}
Fix $a \in X$. For every $t \in \R$, we have
$$\widehat{\mu_a}(t) = \int_K \exp(2\pi i \, \theta(x) \, t) \, d\nu_a(x) = \prod_{n=0}^\infty \bigl(1 - a_n(1-\cos(2 \pi t 2^{-n}))\bigr) \; .$$
Since $a_n \in (0,1/4)$ for all $n$, it follows that for every sequence $t_k \in \R$, the following holds:
$$\widehat{\mu_a}(t_k) \recht 0 \quad\text{if and only if}\quad \sum_{n=0}^\infty a_n(1-\cos(2 \pi t_k 2^{-n})) \recht 0 \quad\text{if and only if}\quad d_a(t_k,0) \recht 0 \; .$$
This concludes the proof of the proposition.
\end{proof}

Combining Proposition \ref{prop.concrete-tau} with the following result, we obtain a more concrete proof for the non-classification of $\tau$-topologies on $\R$.

\begin{proposition}\label{prop.concrete-non-classif-tau}
Define the equivalence relation $\sim$ on $X$ given by $a \sim b$ if and only if the metrics $d_a$ and $d_b$ defined in \eqref{eq.def-metric-da} induce the same topology on $\R$. Then $\sim$ has meager equivalence classes and is generically $S_\infty$-ergodic. In particular, $\sim$ is not classifiable by countable structures.
\end{proposition}
\begin{proof}
Define the equivalence relation $\approx$ on $X$ given by $a \approx b$ if and only if $\sum_{n=0}^\infty |a_n-b_n| < \infty$. Let $a \in X$ and $t_k \in \R$ a sequence such that $d_a(t_k,0) \recht 0$. For every fixed $n$, we have that $a_n > 0$ and thus that $\lim_k d_\Z(2^{-n}t_k)= 0$. It follows that $\approx \, \subset \, \sim$. Define the homeomorphism
$$\vphi : \R^\N \recht X : \vphi(x)_n = \vphi_0(x_n) \quad\text{where}\quad \vphi_0(y) = \frac{1}{8} + \frac{1}{4\pi} \arctan(y) \; .$$
Define the equivalence relation $\approx_1$ on $\R^\N$ given by $a \approx_1 b$ if and only if $a-b \in \ell^1_\R(\N)$. Since $\vphi_0$ is a contraction, when $a \approx_1 b$, also $\vphi(a) \approx \vphi(b)$ and thus, $\vphi(a) \sim \vphi(b)$. Since $\vphi$ is a homeomorphism and since, by \cite[Proposition 3.25]{Hj97}, $\approx_1$ is generically $S_\infty$-ergodic, also $\sim$ is generically $S_\infty$-ergodic.

Let $\cU \subset X$ be a non meager set. By Lemma \ref{lem.a-comeager} below, we can take $a \in \cU$ and a sequence $n_k \in \N$ such that $n_k \recht \infty$ and $d_a(2^{n_k},0) \recht 0$. Also by Lemma \ref{lem.a-comeager} below, we can take $b \in \cU$ such that $d_b(2^{n_k},0) \not\recht 0$. So the topologies on $\R$ induced by $d_a$ and $d_b$ are different. Since $a,b \in \cU$ and $\cU$ was an arbitrary non meager set, it follows that $\sim$ has meager equivalence classes.
\end{proof}

\begin{lemma}\label{lem.a-comeager}
The set of $a \in X$ for which there exists a sequence $n_k \in \N$ such that $n_k \recht \infty$ and $d_a(2^{n_k},0) \recht 0$ is comeager.

Given a sequence $n_k \in \N$ with $n_k \recht \infty$, the set of $b \in X$ such that $d_b(2^{n_k},0) \not\recht 0$ is comeager.
\end{lemma}
\begin{proof}
Note that for all $a \in X$ and $m \in \N$, we have that
$$d_a(2^m,0)^2 = \sum_{k=1}^\infty a_{m+k} 2^{-k} \; .$$
Using that $a_j \in (0,1/4)$ for all $j \in \N$, it follows that for all $a \in X$ and $m,n \in \N$,
\begin{equation}\label{eq.two-ineq}
\frac{1}{2} a_{m+1} \leq d_a(2^m,0)^2 \leq 2^{-n-2} + \sum_{k=1}^n a_{m+k} 2^{-k} \; .
\end{equation}
Using the second inequality in \eqref{eq.two-ineq}, it follows that the first set in the lemma contains
$$\bigcap_{n=1}^\infty \bigcup_{m =n}^\infty \Bigl\{ a \in X \Bigm| \sum_{k=1}^n a_{m+k} 2^{-k} < 2^{-n} \Bigr\} \; ,$$
which is a countable intersection of open dense sets.

Given a sequence $n_k \in \N$ with $n_k \recht \infty$, using the first inequality in \eqref{eq.two-ineq}, it follows that the second set in the lemma contains
$$\bigcap_{k=1}^\infty \bigcup_{j = k}^\infty \bigl\{ b \in X \bigm| b_{n_j+1} > 1/8 \bigr\} \; ,$$
which also is a countable intersection of open dense sets.
\end{proof}

\section{\boldmath Non-classification of actions on free product II$_1$ factors} 

Let $G$ be a locally compact, second countable group and $M$ a von Neumann algebra with separable predual. An \emph{action} of $G$ on $M$ is a continuous group homomorphism from $G$ to the Polish group $\Aut(M)$. We denote this space of actions as $\Hom(G,\Aut(M))$. Equipped with the topology of uniform convergence on compact sets, $\Hom(G,\Aut(M))$ is a Polish space.

There are several natural notions of equivalence of actions. Given $\al,\be \in \Hom(G,\Aut(M))$, one says that
\begin{enumlist}
\item $\al$ and $\be$ are \emph{conjugate} if there exists an element $\theta \in \Aut(M)$ such that $\theta \circ \be_g \circ \theta^{-1} = \al_g$ for all $g \in G$~;
\item $\al$ and $\be$ are \emph{cocycle conjugate} if there exists a continuous map $u : G \recht \cU(M)$ and an element $\theta \in \Aut(M)$ such that
    $$\theta \circ \be_g \circ \theta^{-1} = (\Ad u_g) \circ \al_g \quad\text{and}\quad u_{gh} = u_g \al_g(u_h) \quad\text{for all $g,h \in G$~;}$$
\item $\al$ and $\be$ are \emph{outer conjugate} if there exists an element $\theta \in \Aut(M)$ such that for all $g \in G$, $\theta \circ \be_g \circ \theta^{-1}$ and $\al_g$ have the same image in the outer automorphism group $\Out(M) = \Aut(M)/\Inn(M)$.
\end{enumlist}

Clearly, conjugacy implies cocycle conjugacy, and cocycle conjugacy implies outer conjugacy, but the converse implications need not hold.

We recall from \cite{Sh97} the following definition of the $\tau$-invariant of a group action on a full factor, which is analogous to the $\tau$-invariant for von Neumann algebras (see \cite{Co74a} and the beginning of Section \ref{sec.non-classification-tau-topologies}).

\begin{definition}[{\cite[Definition 8.1]{Sh97}}]
Let $G$ be a locally compact, second countable group and $M$ a full factor with separable predual. For every action $\al \in \Hom(G,\Aut(M))$, the topology $\tau(\al)$ is defined as the weakest topology on $G$ that makes the map $G \recht \Out(M) : g \mapsto \al_g$ continuous.
\end{definition}

Note that $\tau(\al)$ is invariant under outer conjugacy. We can then combine Theorem \ref{thm.non-classifiable-tau} with the construction of Gaussian actions to prove that for noncompact abelian groups $G$, cocycle conjugacy and outer conjugacy of actions of $G$ on the free group factors $L(\F_n)$, and more generally on arbitrary free products $(M,\tau) = (M_1,\tau_1)*(M_2,\tau_2)$ of a diffuse and a nontrivial tracial von Neumann algebra, are not classifiable by countable structures. In particular, outer conjugacy of a single automorphism of $L(\F_n)$ or of such a free product (i.e.\ an action of $G=\Z$) is not classifiable by countable structures. Note that in \cite[Theorem 6.2]{KLP14}, it was proven that conjugacy of a single automorphism of a free product of II$_1$ factors is not classifiable by countable structures.

\begin{theorem}\label{thm.non-class-actions}
Let $G$ be a second countable, locally compact, noncompact, abelian group. Let $(M,\tau) = (M_1,\tau_1)*(M_2,\tau_2)$ be the free product of two von Neumann algebras equipped with a faithful normal tracial state. Assume that $M_1$ is diffuse and that $M_2 \neq \C 1$. Then $M$ is a full II$_1$ factor. None of the equivalence relations of conjugacy, cocycle conjugacy or outer conjugacy on the space $\Hom(G,\Aut(M))$ of actions of $G$ on $M$ is classifiable by countable structures.
\end{theorem}

Before proving Theorem \ref{thm.non-class-actions}, we need the following lemma.

\begin{lemma}\label{lem.gaussian}
Let $(X,\mu)$ be a standard nonatomic probability space and let $G$ be a second countable, locally compact, noncompact, abelian group. Denote by $\Autpmp(X,\mu)$ the Polish group of probability measure preserving automorphisms of $(X,\mu)$. There exists a continuous map
$$\be : \Probc(\Gh) \recht \Hom(G,\Autpmp(X,\mu)) : \rho \mapsto \be_\rho$$
with the following properties.
\begin{enumlist}
\item If $\rho, \eta \in \Probc(\Gh)$ are mutually absolutely continuous, there exists a $\theta \in \Autpmp(X,\mu)$ such that $\be_\rho(g) = \theta \circ \be_\eta(g) \circ \theta^{-1}$ for all $g \in G$.
\item For every $\rho \in \Probc(\Gh)$, the weakest topology on $G$ making the map $G \recht \Autpmp(X,\mu) : g \mapsto \be_\rho(g)$ continuous equals the topology $\tau(\rho)$.
\end{enumlist}
\end{lemma}
\begin{proof}
Let $\cH$ be a fixed separable Hilbert space. Since $\cH$ can be viewed as a real Hilbert space, we can view $(X,\mu)$ as the Gaussian probability space associated with $\cH$, see e.g.\ \cite[Proof of Proposition 1.2]{AEG93}. We then obtain a continuous and injective homomorphism $\pi : \cU(\cH) \recht \Autpmp(X,\mu)$ with the property that for all $u_n,u \in \cU(\cH)$, we have that $u_n \recht u$ strongly if and only if $\pi(u_n) \recht \pi(u)$.

To prove the lemma, we choose a probability measure $\gamma$ on $\Gh$ such that $\gamma \approx \lambdah$, where $\lambdah$ is the Haar measure on $\Gh$. Put $\cH = L^2(\Gh,\gamma)$. Choose a bijective Borel map $\theta : \Gh \recht [0,1]$ such that $\theta_*(\gamma) = \lambda$, where $\lambda$ is the Lebesgue measure on $[0,1]$. Using the notation of Lemma \ref{lem.map-vphi}, define for every $\rho \in \Probc(\Gh)$, the Borel map
$$\psi_\rho : \Gh \recht \Gh : \psi_\rho = \theta^{-1} \circ \vphi_{\theta_*(\rho)} \circ \theta \; .$$
For every $\rho \in \Probc(\Gh)$, we define the unitary representation $\zeta_\rho$ of $G$ on $\cH$ given by
$$\zeta_\rho : G \recht \cU(L^2(\Gh,\gamma)) : (\zeta_\rho(g)\xi)(\om) = (\psi_\rho(\om))(g) \, \xi(\om) \; .$$
Using the unitary $V$ in Lemma \ref{lem.map-vphi}, it follows that the unitary representation $\zeta_\rho$ is unitarily equivalent to the unitary representation $\pi_\rho$ defined in \eqref{eq.rep-pi-mu}. In particular, if $\rho, \eta \in \Probc(\Gh)$ are mutually absolutely continuous, the unitary representations $\zeta_\rho$ and $\zeta_\eta$ are unitarily equivalent. It follows that the map
$$\be : \Probc(\Gh) \recht \Hom(G,\Autpmp(X,\mu)) : \beta_\rho(g) = \pi(\zeta_\rho(g))$$
has all the required properties.
\end{proof}

\begin{proof}[Proof of Theorem \ref{thm.non-class-actions}]
By \cite[Theorem 3.7]{Ue10}, any free product $(M,\tau) = (M_1,\tau_1)*(M_2,\tau_2)$ of tracial von Neumann algebras with $M_1$ diffuse and $M_2$ nontrivial is a full II$_1$ factor. Moreover, in the proof of \cite[Theorem 3.7]{Ue10}, the following is shown: if $N_1 \subset M_1$ is a diffuse von Neumann subalgebra and if $u_n \in M$ is a bounded sequence such that $\|u_n x - x u_n\|_2 \recht 0$ for all $x \in N_1 \cup M_2$, then $\|u_n - \tau(u_n)1\|_2 \recht 0$.

We are in exactly one of the following two cases.

{\bf Case 1.} $M_1$ has a direct summand of type I$_n$.

{\bf Case 2.} $M_1$ is of type II$_1$.

{\bf Proof in case 1.} Take a nonzero central projection $z \in \cZ(M_1)$ such that $M_1 z = A \ot M_n(\C)$ for some $n \geq 1$ and some diffuse abelian von Neumann algebra $A$. Write $A = L^\infty(X,\mu) \ovt B$ where $B$ is a diffuse abelian von Neumann algebra. Whenever $\be \in \Autpmp(X,\mu)$, we denote by $\Psi(\be) \in \Aut(M_1)$ the unique automorphism given by the identity on $N_1 = (B \ot M_n(\C)) \oplus M_1(1-z)$ and given by $\be$ on $L^\infty(X,\mu) \subset M_1 z$. Note that $\be \mapsto \Psi(\be)$ is a continuous homomorphism from $\Autpmp(X,\mu)$ to the Polish group of trace preserving automorphisms of $(M_1,\tau_1)$. Using the notation of Lemma \ref{lem.gaussian}, define
$$\Theta : \Probc(\Gh) \recht \Hom(G,\Aut(M)) : \rho \mapsto \Theta_\rho \quad\text{with}\;\; \Theta_\rho(g) = \Psi(\be_\rho(g)) * \id \; .$$
By Lemma \ref{lem.gaussian}, we get that the actions $\Theta_\rho$ and $\Theta_\eta$ are conjugate whenever the probability measures $\rho,\eta \in \Probc(\Gh)$ are mutually absolutely continuous.

We claim that the $\tau$-invariant $\tau(\Theta_\rho)$ equals $\tau(\rho)$. To prove the claim, assume that $\Theta_\rho(g_n) \recht \id$ in $\Out(M)$. We have to prove that $\Theta_\rho(g_n) \recht \id$ in $\Aut(M)$, because then $\be_\rho(g_n) \recht \id$ in $\Autpmp(X,\mu)$, so that $g_n \recht e$ in $\tau(\rho)$ by Lemma \ref{lem.gaussian}. Take $u_n \in \cU(M)$ such that $(\Ad u_n) \circ \Theta_\rho(g_n) \recht \id$ in $\Aut(M)$. Since the automorphisms $\Theta_\rho(g_n)$ act as the identity on $N = N_1 * M_2$, we get that $\|u_n x - x u_n\|_2 \recht 0$ for all $x \in N$. By \cite[Theorem 3.7]{Ue10}, as explained in the first paragraph of the proof, we get that $\Ad u_n \recht \id$ and the claim follows.

We can now complete the proof of the theorem in case 1. Assume that $f : \Hom(G,\Aut(M)) \recht Z$ is a Borel reduction of either conjugacy, cocycle conjugacy or outer conjugacy to the orbit equivalence relation $E$ of a continuous action of $S_\infty$ on a Polish space $Z$. Since $G$ is noncompact, the dual group $\Gh$ is nondiscrete and therefore, $\Gh$ has no isolated points. By \cite[Theorem 2]{KS99}, the equivalence relation $\approx$ of mutual absolute continuity on $\Prob(\Gh)$ is generically $S_\infty$-ergodic. By \cite[Proposition 4.1]{KS99}, $\Probc(\Gh) \subset \Prob(\Gh)$ is a dense $G_\delta$-set. So $\approx$ is also generically $S_\infty$-ergodic on $\Probc(\Gh)$.

Since $\Theta_\rho$ and $\Theta_\eta$ are conjugate when $\rho \approx \eta$, we get that $(f(\Theta_\rho), f(\Theta_\eta)) \in E$ whenever $\rho \approx \eta$. So, there exists a comeager set $\cU \subset \Probc(\Gh)$ such that $(f(\Theta_\rho) , f(\Theta_\eta)) \in E$ for all $\rho,\eta \in \cU$. By Theorem \ref{thm.non-classifiable-tau}, equality of $\tau$-topologies on $\Prob(\Gh)$ has meager equivalence classes. So, we find $\rho,\eta \in \cU$ with $\tau(\rho) \neq \tau(\eta)$. Since $(f(\Theta_\rho), f(\Theta_\eta)) \in E$ and $f$ is a Borel reduction, we also get that $\Theta_\rho$ and $\Theta_\eta$ are outer conjugate. This implies that $\tau(\Theta_\rho) = \tau(\Theta_\eta)$. So, $\tau(\rho) = \tau(\eta)$ and we have reached a contradiction.

{\bf Proof in case 2.} Fix a projection $q \in M_2$ with $0 < \tau(q) < 1$. By \cite[Proposition 3.2]{Dy92b}, we can take a large enough integer $n \geq 1$ such that the free product $(M_n(\C),\tau)*(\C q + \C (1-q),\tau)$ is of type II$_1$. A fortiori, $(M_n(\C),\tau)*(M_2,\tau)$ is of type II$_1$. Since $M_1$ is of type II$_1$, we can write $M_1 = N_1 \ot M_n(\C)$. Fix a minimal projection $p \in M_n(\C)$. Then, $M \cong pMp \ot M_n(\C)$. By \cite[Theorem 1.2]{Dy92b}, we have $pMp \cong N_1 * N_2$ where $N_2 = p(M_n(\C) * M_2)p$. Now we can use the same trick as in \cite[Theorem 6.2]{KLP14}. Since $N_1$ and $N_2$ are of type II$_1$, we can write $N_i = P_i \ot M_2(\C)$.
By \cite[Theorem 3.5(iii)]{Dy92a}, we can identify
$$N_1 * N_2 = (P_1 \ot M_2(\C)) * (P_2 \ot M_2(\C)) = (L(\F_3) * P_1 * P_2) \ot M_2(\C) \; .$$
Further identifying $L(\F_3) = L^\infty(X,\mu) * L(\F_2)$, we have found an isomorphism
$$\theta : M \recht (L^\infty(X,\mu) * L(\F_2) * P_1 * P_2) \ot M_{2n}(\C) \; .$$
Defining
\begin{multline*}
 \Theta : \Probc(\Gh) \recht \Hom(G,\Aut(M)) : \rho \mapsto \Theta_\rho \quad\text{with}\\  \Theta_\rho(g) = \theta^{-1} \circ \bigl((\be_\rho(g) * \id * \id * \id) \ot \id\bigr) \circ \theta \; ,
\end{multline*}
the same reasoning as above concludes the proof of the theorem.
\end{proof}

\begin{remark}
In the specific case of $M=L(\F_\infty)$, one can also consider the so-called free Bogoljubov actions arising from Voiculescu's free Gaussian functor (see \cite{Vo83}). When $G$ is a second countable, locally compact, noncompact, abelian group, the same arguments as in the proof of Theorem \ref{thm.non-class-actions} show that the equivalence relations of conjugacy, cocycle conjugacy and outer conjugacy on the free Bogoljubov actions are not classifiable by countable structures.
\end{remark}

\end{document}